\documentclass[12pt]{amsart}

\usepackage[margin = 3cm]{geometry}

\usepackage{amsmath}
\usepackage{amssymb}
\usepackage{amsthm}
\usepackage{amscd}
\usepackage{delarray}
\usepackage{enumitem}
\usepackage{colonequals}

\headsep=1cm \numberwithin{equation}{section}
\newtheorem{theorem}{Theorem}[section]
\newtheorem{lemma}[theorem]{Lemma}
\newtheorem{proposition}[theorem]{Proposition}

\theoremstyle{definition}

\newtheorem{convention and reminder}[theorem]{Convention and Reminder}
\newtheorem{convention and remark}[theorem]{Convention and Remark}
\newtheorem{definition and remark}[theorem]{Definition and Remark}

\newtheorem{reminders and definition}[theorem]{Reminders and Definition}
\newtheorem{notation}[theorem]{Notation}
\newtheorem{notation and remarks}[theorem]{Notation and Remarks}
\newtheorem{notation and remark}[theorem]{Notation and Remark}
\newtheorem{example}[theorem]{Example}

\begin{document}

\title{Secant rank and syzygies of projections of elliptic normal curves}

\author[C. Han]{Changho Han}
\email{changho\_han@korea.ac.kr}
\address{Changho Han, Department of Mathematics, Korea University, Seoul 02841, Republic of Korea}

\author[E. Park]{Euisung Park}
\email{euisungpark@korea.ac.kr}
\address{Euisung Park, Department of Mathematics, Korea University, Seoul 02841, Republic of Korea}


\begin{abstract}
We study the syzygies of projections of elliptic normal curves. 
Let $C \subset \mathbb{P}^{d-1}$ be an elliptic normal curve of degree $d \ge 5$, and let $C_q$ denote the projection of $C$ from a point $q$. 
We obtain sharp bounds for the Green--Lazarsfeld index of $C_q$ in terms of the secant rank of $q$. 
More precisely, if $q \in C^s \setminus C^2$, where $C^s$ is the $s$-th secant variety of $C$, then $\mathrm{index}(C_q) \le s-3$, and equality holds for a general point $q$ of $C^s$.
In particular, $\mathrm{index}(C_q) = \lceil \frac{d}{2} \rceil - 3$ for a general point $q$ in $\mathbb{P}^{d-1}$.
The proof realizes projected elliptic curves as hyperplane sections of elliptic ruled surface scrolls and exploits the known syzygetic properties of these scrolls.
\end{abstract}

\maketitle \thispagestyle{empty}

\section{Introduction}
\noindent Throughout this paper we work over an algebraically closed field $\mathbb{K}$ of characteristic zero. 

The modern study of syzygies of projective varieties began with the foundational work of Mark Green \cite{G}, who introduced Koszul cohomology as a tool for understanding minimal free resolutions of homogeneous coordinate rings. 
For a projective variety $X \subset \mathbb{P}^r$, one says that $X$ satisfies property $N_{2,p}$ if the homogeneous ideal of $X$ is generated by quadrics and the syzygies among them are generated by linear syzygies up to the $p$-th step. 
The maximal such integer $p$ is called the \emph{Green--Lazarsfeld index} of $X$ and is denoted by $\mathrm{index}(X)$;
by convention, $\mathrm{index}(X)=0$ if $X$ fails to satisfy property $N_{2,1}$.
It is well known that
\[
\mathrm{index}(X) < \delta (X),
\]
where $\delta (X)$ denotes the largest integer $p$ such that $X$ admits a $(p+2)$-secant $p$-plane (cf. \cite[Theorem 1.1]{EGHP}).

Now let $C \subset \mathbb{P}^{d-g}$ be a linearly normal smooth projective curve of genus $g$ and degree $d \ge 2g+2$.
Green \cite{G} proved that $C$ satisfies property $N_{2,d-2g-1}$.
Green and Lazarsfeld \cite{GL} gave another proof and characterized the extremal case.
In particular, $\mathrm{index}(C) \ge d-2g-1$, and equality holds for all $d \geq 3g-2$. 
Recent developments have further clarified the role of secant geometry in 
syzygies of projective curves. 
The Green--Lazarsfeld secant conjecture predicts that the failure of property $N_p$ for a curve embedded by a sufficiently positive line bundle is governed by the existence of $(p+2)$-secant $p$-planes. 
This conjecture has been confirmed in several cases: Farkas and Kemeny \cite{FarkasKemeny} proved the generic version of the Green--Lazarsfeld secant conjecture, while Kemeny \cite{Kemeny} established the extremal case for curves of arbitrary gonality.

Another natural problem is to understand how the syzygies of a projective curve change under linear projections. 
For $q \in \mathbb{P}^{d-g} \setminus C^2$, let
\[
C_q \colonequals \pi_q(C) \subset \mathbb{P}^{d-g-1}
\]
be the projection of $C$ from $q$.
When $d \geq 2g+3$, $C_q$ is always $3$-regular by \cite[Theorem 1]{Noma}, so its homogeneous ideal is generated by quadratic and cubic polynomials. 
It is therefore natural to ask: how does the Green--Lazarsfeld index of $C_q$ depend on the position of the projection center $q$? 

To measure the position of $q$ relative to $C$, we use secant varieties. 
Recall that the $s$-th secant variety $C^s$ of $C$ is the Zariski closure of the union of linear spans of $s$ general points on $C$;
these varieties form a strictly increasing filtration
\[
C \subset C^2 \subset C^3 \subset \cdots \subset C^{\left \lceil \frac{d-g+1}{2}  \right \rceil} = \mathbb{P}^{d-g}
\] 
with $\dim C^s = 2s-1$. 
For every $q \in \mathbb{P}^{d-g}$,
the \emph{secant rank} of $q$ with respect to $C$ is defined by
\[
\mathrm{rk}_C (q) \colonequals \min \{ s \mid q \in C^s \}.
\]
Suppose that $q$ is a general point of $C^s$ for some $s \ge 3$. Then $\mathrm{rk}_C (q) =s$, and $q$ is contained in an $(s-1)$-plane spanned by $s$ points of $C$. The latter property implies that the projected curve $C_q$ admits an $s$-secant $(s-2)$-plane and hence
\[
\mathrm{index}(C_q) \leq s-3.
\]
Combining this observation with upper semicontinuity of graded Betti numbers, we obtain the following result.

\begin{theorem}\label{thm:upper bound of index}
Let $C \subset \mathbb{P}^{d-g}$ be a linearly normal smooth projective curve of genus $g$ and degree $d \geq 2g+3$, and let $q \in \mathbb{P}^{d-g} \setminus C^2$ be a point. Then
\[
\mathrm{index}(C_q) \leq \mathrm{rk}_C (q)-3.
\]
\end{theorem}

When $C \subset \mathbb{P}^d$ is a rational normal curve, the second-named author showed in \cite{P07} that the inequality of Theorem 1.1 is an equality: for every $q \in \mathbb{P}^d \setminus C^2$,
\[
\mathrm{index}(C_q)=\mathrm{rk}_C (q)-3.
\]
It is then natural to ask whether such an equality extends to curves of positive genus.
Partial results were obtained by Lee and Park \cite{LP13}, who showed that for a fixed $p$ and sufficiently large $d$, a general projection $C_q$ satisfies property $N_{2,p}$.

Our main theorem shows that for elliptic normal curves, the equality holds for a general projection center:

\begin{theorem}\label{thm:main}
Let $C \subset \mathbb{P}^{d-1}$ be an elliptic normal curve of degree $d \ge 5$, and let $q \in \mathbb{P}^{d-1} \setminus C^2$ be a point with $s \colonequals \mbox{\rm rk}_{C} (q)$. Then
\begin{enumerate}[leftmargin=7mm, labelsep=1.5mm, itemsep=0pt, topsep=2pt, font=\upshape]
	\item ${\rm index} (C_q) \leq s-3$.
	\item If $q$ is a general point of $C^s$, then ${\rm index} (C_q) = s-3$.
	\item If $q$ is a general point of $\mathbb{P}^{d-1}$, then ${\rm index} (C_q) = \lceil \frac{d}{2}\rceil -3$.
\end{enumerate}
\end{theorem}

Thus the Green--Lazarsfeld index of a projected elliptic normal curve is governed by the secant rank of the projection center, paralleling the behavior known for rational normal curves.

The key idea of the proof, carried out in Section~\ref{sec:proj_ell_normal_curves}, is to realize projected elliptic curves as hyperplane sections of suitable elliptic ruled surface scrolls and to exploit the known syzygetic properties of these scrolls.

In Section~\ref{sec:rmk_proj_higher_dim_vars}, we observe that for higher-dimensional smooth projective varieties, the analog of Theorem~\ref{thm:main} is not true in general. 
We raise the question of for which higher-dimensional smooth projective varieties the increasing filtration by higher secant varieties reflects the behavior of the Green–Lazarsfeld index of projections.

\medskip

\subsection*{Acknowledgements} 
The authors would like to thank Jong-In Han for helpful comments and conversations.
The first-named author is supported by Korea University grants and the National Research Foundation of Korea (NRF) grant funded by the Korea government (MSIT) (No. RS-2025-24535254). 
The second-named author is supported by the National Research Foundation of Korea (NRF) grant funded by the Korea government (MSIT) (No. 2022R1A2C1002784).  

\section{Preliminaries}
\noindent We review some background material on Koszul cohomology and upper semicontinuity of graded Betti numbers. 
Experts will find no new materials in this section.

\subsection{Koszul cohomological method}
We recall Green's cohomological interpretation \cite{G} of graded Betti numbers of graded modules associated to a coherent sheaf.

\begin{lemma}\label{lem:exactsequence}
For a nonzero coherent sheaf $\mathcal{F}$ on $\mathbb{P}^r$, let
\begin{equation*}
F = \bigoplus_{j \in \mathbb Z} H^0 (\mathbb{P}^r, \mathcal{F}(j))
\end{equation*}
be the graded $S$-module associated to $\mathcal{F}$, where $S$ is the homogeneous coordinate ring of $\mathbb{P}^r$. Then there is an exact sequence
\begin{equation*}
0 \rightarrow \mbox{Tor}^S _{i} (F,\mathbb{K})_{i+j} \rightarrow H^1 (\mathbb{P}^r,\bigwedge^{i+1} \mathcal{M} \otimes \mathcal{F} (j-1) ) \rightarrow \bigwedge^{i+1} V \otimes H^1 (\mathbb{P}^r, \mathcal{F} (j-1))
\end{equation*}
\begin{equation*}
\quad \quad \quad \rightarrow  H^1 (\mathbb{P}^r,\bigwedge^{i} \mathcal{M} \otimes \mathcal{F} (j) ) \rightarrow H^2 (\mathbb{P}^r,\bigwedge^{i+1} \mathcal{M} \otimes \mathcal{F} (j-1) ) \rightarrow \cdots,
\end{equation*}
where $\mathcal{M}=\Omega _{\mathbb{P}^r}(1)$ and $V=H^0
(\mathbb{P}^r,\mathcal{O}_{\mathbb{P}^r}(1))$.
\end{lemma}

\begin{proof}
See \cite[Theorem 5.8]{E}.
\end{proof}

\begin{lemma}\label{lem:3regular}
Let $X \subset \mathbb{P}^r$ be a $3$-regular variety, that is, $H^i (\mathbb{P}^r,\mathcal{I}_X(i-3))=0$ for all $i \geq 1$. 
Then for every positive integer $p \le \operatorname{codim}(X,\mathbb{P}^r)$, $X$ satisfies property $N_{2,p}$ if and only if
\[
H^1(\mathbb{P}^r,\bigwedge^p \mathcal{M} \otimes \mathcal{I}_X(2))=0.
\]
\end{lemma}

\begin{proof}
See \cite[Lemma 1.4]{GL} and \cite[Lemma 2.1]{P07}.
\end{proof}

\subsection{Upper Semicontinuity of the Betti Numbers}
In this subsection, we establish an upper semicontinuity result for certain graded Betti numbers arising from families of projective schemes. Let us first fix some notations.

\begin{notation}\label{not:secant varieties}
Let $X \subset \mathbb{P}^r$ be a nondegenerate smooth projective variety.
\begin{enumerate}[leftmargin=7mm, labelsep=1.5mm, itemsep=0pt, topsep=2pt, font=\upshape]
	\item For each $k \geq 2$, let $X^k$ denote the $k$-\textit{th secant variety} of $X$, i.e., the closure of the set of points lying in the $(k-1)$-dimensional linear subspaces spanned by general collections of $k$ points of $X$.
	\item The number ${\rm ord} (X) = \min \{ k ~|~ X^k = \mathbb{P}^r \}$ is called the \textit{order} of $X$.
	\item Set $U_k = X^k \setminus X^{k-1}$ and $U = \mathbb{P}^r \setminus X^2$.
	\item For $q \in U$, let $X_q \subset \mathbb{P}^{r-1}$ denote the projection of $X$ from $q$.
	\item For each $3 \leq k \leq {\rm ord} (X)$, define
	\begin{equation*}
		m_k (X) \colonequals \max \{ {\rm index} (X_q ) ~|~ q \in U_k \}.
	\end{equation*}
\end{enumerate}
\end{notation}

\begin{proposition}\label{prop:upper semicontinuity}
Let $X \subset \mathbb{P}^r$ be as in Notation~\ref{not:secant varieties}. Assume that $3 \le k \le {\rm ord}(X)$ and $X_q$ is $3$-regular for every $q \in U_k$. If $q$ is a general point in $U_k$, then
\begin{equation*}
{\rm index} (X_q ) = m_k (X).
\end{equation*}
\end{proposition}

\begin{proof}
When $k \ge 3$ and $q \in U_k$, the two varieties $X_q \subset \mathbb{P}^{r-1}$ and ${\rm join}(q,X) \subset \mathbb{P}^r$ have the same graded Betti numbers. Thus, it suffices to prove that
\begin{equation*}
	{\rm index}({\rm join}(q,X)) = m_k (X)
\end{equation*}
for a general point $q$ in $U_k$.

To construct a family over $U_k$ whose fiber over $q \in U_k$ is ${\rm join}(q,X)$, first consider an incidence subvariety
\begin{equation*}
\widetilde{M} \colonequals \{(x,y,z) \in U_k \times \mathbb{P}^r \times X \mid y \in \langle x,z\rangle \} \subset U_k \times \mathbb{P}^r \times X,
\end{equation*}
which is a $\mathbb{P}^1$-bundle over $U_k \times X$.
Let
\begin{equation*}
\pi_{12} : U_k \times \mathbb{P}^r \times X \to U_k \times \mathbb{P}^r
\end{equation*}
be a projection. Since $\pi_{12}$ is proper over $U_k$, its image
\begin{equation*}
M \colonequals \pi_{12}(\widetilde{M})  = \{(x,y) \mid  \langle x,y \rangle \cap X \neq \emptyset \} \subset U_k \times \mathbb{P}^r
\end{equation*}
is projective over $U_k$ and $M_q = {\rm join}(q,X)$ for every $q \in U_k$.

Now, let $\mathcal{F} = \mathcal{I}_{M}$ be the ideal sheaf of $M \subset U_k \times \mathbb{P}^r$;
then for every $q \in U_k$,
\begin{equation*}
\mathcal{F}_q = \mathcal{I}_{M_q} = \mathcal{I}_{{\rm join}(q,X)}
\end{equation*}
on $\mathbb{P}^r$. 
Let $F_q = \oplus_{j \in \mathbb{Z}} H^0(\mathbb{P}^r, \mathcal{F}_q(j))$ be the associated graded $S$-module.
Since $X_q$ is $3$-regular and $q \not \in X^2$, $M_q$ is also $3$-regular. So
\[
H^1\!\left(\mathbb{P}^r, \mathcal{I}_{M_q}(2)\right)=0.
\]
Hence Lemma~\ref{lem:exactsequence} implies that
\[
\operatorname{Tor}^S_i(F_q,\mathbb{K})_{i+3}
\cong H^1\!\left( \mathbb{P}^r, \bigwedge^{i+1}\mathcal{M}_q \otimes \mathcal{I}_{M_q}(2) \right),
\]
where $\mathcal{M} (-1) = \Omega^1 _{U_k \times \mathbb{P}^r/U_k}$ is the sheaf of relative K\"ahler differentials on the projection $\pi \colon U_k \times \mathbb{P}^r \to U_k$; observe that $\mathcal{M}_q = \Omega_{\mathbb{P}^r} (1)$. Since $\pi$ is projective and $M$ in $U_k \times \mathbb{P}^r$ is flat over $U_k$ by \cite[Theorem III.9.9]{Hart}, applying \cite[Corollary III.12.8]{Hart} to $\bigwedge^{i+1}\mathcal{M} \otimes \mathcal{I}_{M}(2)$ implies that the graded Betti numbers
\[
\beta_{i,2}(X_q) = \beta_{i,2}(M_q) = \dim_{\mathbb{K}} \operatorname{Tor}^S_i(\mathcal{F}_q,\mathbb{K})_{i+3} = h^1\!\left( \mathbb{P}^r, \bigwedge^{i+1}\mathcal{M}_q
\otimes \mathcal{I}_{M_q}(2) \right)
\]
define an upper semicontinuous function of $q$ in $U_k$.
\end{proof}

\section{Projections of elliptic normal curves} \label{sec:proj_ell_normal_curves}
\noindent This section aims to prove the following theorem.

\begin{theorem}\label{thm:elliptic}
Let $C \subset \mathbb{P}^{d-1}$ be an elliptic normal curve of degree $d \geq 5$. 
Then for each integer $s \in \{3,\ldots, \lceil\frac{d}{2}\rceil \}$, there exists a point $q \in \mathbb{P}^{d-1}$ such that
\begin{equation*}
{\rm rk}_{C} (q ) = s \quad \mbox{and} \quad {\rm index} (C_q)=s-3.
\end{equation*}
\end{theorem}

We first prove Theorem~\ref{thm:upper bound of index}, which holds for curves of arbitrary genus, before specializing to the elliptic curves in the proof of Theorem~\ref{thm:elliptic}.

\begin{proof}[{\bf Proof of Theorem~\ref{thm:upper bound of index}}]
Let $s$ be any fixed integer in $\{3,\ldots, \lceil\frac{d-g+1}{2}\rceil \}$. 
First, suppose that $q$ is a general point in $C^s$. 
Then, there exist points $q_1 , \ldots , q_s \in C$ such that
\begin{equation*}
\Lambda \colonequals \langle q_1 , \ldots , q_s \rangle
\end{equation*}
is an $(s-1)$-plane containing $q$. 
Hence $\pi_q (\Lambda)$ is an $s$-secant $(s-2)$-plane to $C_q$. By \cite[Theorem 1.1]{EGHP}, the curve
\begin{equation*}
C_q \subset \mathbb{P}^{d-g-1}
\end{equation*}
fails to satisfy property $N_{2,s-2}$. 
This implies that ${\rm index}(C_q) \le s-3$ for a general point $q \in C^s$. 
Since $d \ge 2g+3$, $C_q$ is $3$-regular for every $q \in \mathbb{P}^{d-g} \setminus C^2$ by \cite[Theorem 1]{Noma}.
Hence Proposition~\ref{prop:upper semicontinuity} implies that ${\rm index} (C_q) \leq s-3 = \mathrm{rk}_C (q)-3$ for every $q \in C^s \setminus C^{s-1}$.
\end{proof}

To prove Theorem~\ref{thm:elliptic}, we investigate the minimal free resolutions of unisecant divisors on ruled surface scrolls.

\begin{notation and remarks}\label{elliptic ruled surface}
Let $C$ be an elliptic curve, and let $\mathcal{E}$ be a rank two vector bundle on $C$. Let $S = \mathbb{P}_C (\mathcal{E})$ be the associated ruled surface with projection morphism $\pi : S \rightarrow C$. We follow the notation and terminology of \cite[Chapter V, \S~2]{Hart}.
\begin{enumerate}[leftmargin=7mm, labelsep=1.5mm, itemsep=0pt, topsep=2pt, font=\upshape]
	\item We assume that $\mathcal{E}$ is normalized in the sense that $H^0 (C,\mathcal{E})\neq 0$ while $H^0 (C,\mathcal{E} \otimes \mathcal{O}_C (D))=0$ for every divisor $D$ on $C$ of negative degree. We set
	\begin{equation*}
		\mathfrak{e}=\wedge^2 \mathcal{E}~~~~\mbox{and}~~~~e = -\deg (\mathfrak{e}).
	\end{equation*}
	The invariant $e$ of $S$ satisfies $e \geq -1$, and $e=-1$ if and only if $\mathcal{E}$ is stable. We fix a minimal section $C_0$ such that $\mathcal{O}_S (C_0)$ is the tautological line bundle of $S$. For $\mathfrak{b} \in \mbox{Pic} (C)$, we denote by $\mathfrak{b}\mathfrak{f}$ the pullback $\pi^*\mathfrak{b}$. Thus every line bundle on $S$ can be written uniquely as $\mathcal{O}_S (aC_0+\mathfrak{b}\mathfrak{f})$ for some $a\in \mathbb{Z}$ and $\mathfrak{b} \in \mbox{Pic} (C)$. Also, every element of $\mbox{Num}(S)$ can be written $aC_0 +b\mathfrak{f}$ with $a,b \in \mathbb{Z}$.
	
	\item Let $L= \mathcal{O}_S (C_0 + \mathfrak{b} \mathfrak{f})$ be a line bundle on $S$ with $\mbox{deg}(\mathfrak{b})=b \geq e+3$. Then $L$ is a very ample line bundle, defining a $3$-regular projectively normal embedding (cf. \cite{P06})
	\begin{equation*}
		S \subset \mathbb{P}^r , ~ r= 2b-e-1.
	\end{equation*}
	Furthermore, \cite[Theorem 1.4]{P06} shows that $S \subset \mathbb{P}^r$ satisfies property $N_{2,p}$ if and only if $b \geq e+3+p$.
\end{enumerate}
\end{notation and remarks}

\begin{lemma}\label{lem:construction}
Let $S \subset \mathbb{P}^r$ be as in Notation~and~Remarks~\ref{elliptic ruled surface}.$(2)$, and let $X \subset \mathbb{P}^{r-1}$ be a general hyperplane section of $S$. Then
\begin{enumerate}[leftmargin=7mm, labelsep=1.5mm, itemsep=0pt, topsep=2pt, font=\upshape]
	\item ${\rm index}(X) = b-e-3$.
	\item $X$ is equal to $C_q \colonequals \pi_q (C)$, where
	\begin{equation*}
		C \subset \mathbb{P}^r
	\end{equation*}
	is the elliptic normal curve embedded by the line bundle $\mathfrak{b}^2 \otimes \mathfrak{e}$ and $q$ is a point in $\mathbb{P}^r$ with $\mbox{\rm rk}_{C} (q) = b-e$.
\end{enumerate}
\end{lemma}

\begin{proof}
$(1)$ Since $S$ is projectively normal and ${\rm index}(S) = b-e-3$ (cf. Notation~and~Remarks~\ref{elliptic ruled surface}.$(2)$), a general hyperplane section $X$ satisfies ${\rm index} (X) = b-e-3$.

\noindent $(2)$ Let $\mathcal{L}$ be the restriction of $L= \mathcal{O}_S (C_0 + \mathfrak{b} \mathfrak{f})$ to $X$. Note that $X \cong C$, and the degree of $X$ in $\mathbb{P}^{r-1}$ is equal to $r+1$, the degree of $S$. This implies that $X$ is equal to $C_q \colonequals \pi_q (C)$ for some $q \in \mathbb{P}^r \setminus C^2$, where $C \subset \mathbb{P}^r$ is the elliptic normal curve embedded by $\mathcal{L}$. Tensoring $L$ to the exact sequence
\begin{equation*}
0 \rightarrow \mathcal{O}_S (-X) \rightarrow \mathcal{O}_S \rightarrow \mathcal{O}_X \rightarrow 0,
\end{equation*}
we get
\begin{equation*}
0 \rightarrow \mathcal{O}_S \rightarrow L \rightarrow L|_X \rightarrow 0.
\end{equation*}
By pushing forward via $\pi$, we obtain
\begin{equation*}
0 \rightarrow \mathcal{O}_C \rightarrow \mathcal{E} \otimes \mathfrak{b}  \rightarrow L|_X \rightarrow 0.
\end{equation*}
This implies that $L|_X$ is equal to $\mathfrak{b}^2 \otimes \mathfrak{e}$, the determinant of $\mathcal{E} \otimes \mathfrak{b}$.

Finally, to prove the equality $\mbox{\rm rk}_{C} (q) = b-e$, we consider the scheme-theoretic intersection $\Gamma \colonequals C_0 \cap X$. Note that
\begin{equation*}
| \Gamma | = C_0 \cdot (C_0 +\mathfrak{b} \mathfrak{f} ) =  b-e.
\end{equation*}
Since the linear series $|L|$ restricts to a nondegenerate embedding of $C_0$ into $\mathbb{P}^{b-e-1}$, the hyperplane section $\Gamma$ of $C_0$ spans $\mathbb{P}^{b-e-2}$. Therefore, $\mbox{\rm rk}_{C} (q) \leq b-e$ and hence $q \in C^{b-e}$. On the other hand, if $q$ is contained in $C^{b-e-1}$, then ${\rm index} (X) < b-e-3$ by Theorem~\ref{thm:upper bound of index}, which contradicts $(1)$. Consequently, $q \notin C^{b-e-1}$, and hence $\mbox{\rm rk}_{C} (q) = b-e$.
\end{proof}

\begin{proof}[{\bf Proof of Theorem~\ref{thm:elliptic}}]
Let $S = \mathbb{P}_C (\mathcal{E})$ be a ruled surface over $C$ associated to a normalized rank two vector bundle $\mathcal{E}$ on $C$ with $\mathfrak{e}=\wedge^2 \mathcal{E}$ such that
\begin{equation*}
e = d-2s.
\end{equation*}
Indeed, $3 \leq s \leq \lceil \frac{d}{2}\rceil$ and hence $e \geq -1$. Thus, there does exist such a surface $S$. Now, let $\mathfrak{b} \in \mbox{Pic}(C)$ be a line bundle of degree $d-s$ such that
\begin{equation*}
\mathfrak{b}^2 = \mathcal{O}_C (1) \otimes \mathfrak{e}^{-1} .
\end{equation*}
Again, there exists such a line bundle $\mathfrak{b}$ since $\mbox{Pic}^{0} (C)$ is $2$-divisible. Moreover,
\begin{equation*}
b = d-s = e+s \geq e+3
\end{equation*}
and hence
\begin{equation*}
L \colonequals \mathcal{O}_S (C_0 + \mathfrak{b} \mathfrak{f})
\end{equation*}
is a very ample line bundle on $S$, defining a $3$-regular projectively normal embedding
\begin{equation*}
S \subset \mathbb{P}^{d-1},
\end{equation*}
whose Green--Lazarsfeld index is equal to $b -e-3 = s-3$. For details, see Notation~and~Remarks~\ref{elliptic ruled surface}.$(2)$. By Lemma~\ref{lem:construction}, there is a point $q$ in $C^{b-e}$ such that
\begin{equation*}
{\rm index}(C_q ) = b-e-3 = s-3.
\end{equation*}
This implies that $\mbox{\rm rk}_{C} (q) = s$.
Indeed, if $\mbox{\rm rk}_{C} (q) < s$, then ${\rm index}(C_q ) < s-3$ by Theorem~\ref{thm:upper bound of index}. This completes the proof.
\end{proof}

Now, we are ready to give a proof of Theorem~\ref{thm:main}.

\begin{proof}[{\bf Proof of Theorem~\ref{thm:main}}]
$(1)$ This follows immediately from Theorem \ref{thm:upper bound of index}.

\noindent $(2)$ By Theorem~\ref{thm:elliptic}, there exists a point $q \in \mathbb{P}^{d-1}$ such that
\begin{equation*}
{\rm rk}_{C} (q ) = s \quad \mbox{and} \quad {\rm index} (C_q)=s-3.
\end{equation*}
Thus the assertion comes by combining $(1)$ and Proposition~\ref{prop:upper semicontinuity}.

\noindent $(3)$ This comes immediately by applying $(2)$ to the case where
\begin{equation*}
s = {\rm ord} ( C) = \left \lceil  \frac{d}{2} \right \rceil.
\end{equation*}
\end{proof}

\section{Remarks on projection of higher dimensional varieties} \label{sec:rmk_proj_higher_dim_vars}

\noindent Let $X \subset \mathbb{P}^r$ be a linearly normal smooth projective variety of dimension $n$, and let $q$ be a point in $\mathbb{P}^r \setminus X^2$.

If $X$ is a curve, then Theorem 1.1 in \cite{P07}, Theorem 1.1 in \cite{LP13}, and Theorem~\ref{thm:main} in the present paper show that as the position of $q$ relative to $X$ becomes more general, the Green--Lazarsfeld index of $X_q$ increases.
It is therefore natural to ask whether this phenomenon also persists when $n \geq 2$. However, the following example shows that this cannot be expected to hold for all $X$.

\begin{example}\label{ex:projection of scroll}
	Let $X = S(1,a_2,\ldots,a_n ) \subset \mathbb{P}^r$ be an $n$-dimensional smooth rational normal scroll, where $1=a_1 \leq a_2 \leq \cdots \leq a_n$ is a nondecreasing sequence of integers. Also, let $q$ be a point in $\mathbb{P}^r \setminus X^2$.
	
	By Theorem 1.4 in \cite{BS07} and Theorem 1.1 in \cite{P07b}, there exists a smooth $(n+1)$-fold rational normal scroll $Y = S(b_1 , \ldots , b_n , b_{n+1} )$ in $\mathbb{P}^{r-1}$ containing $X_q$. Theorems 1.2 and 1.3 in \cite{P07b} show that $b_1 =1$ and the Green--Lazarsfeld index of $X_q$ is $0$.
	
	This shows that even though ${\rm index}(X) = \infty$ and ${\rm rk}_X (q)$ can be arbitrarily large, one necessarily has ${\rm index}(X_q ) =0$ for all $q \in \mathbb{P}^r \setminus X^2$.
\end{example}

From this example, we see that the phenomenon observed for curves cannot, in general, be expected for higher dimensional projective varieties. 
Note that this example is special because $X$ itself is a variety of minimal degree. On the other hand, there are also positive results.

\begin{example}\label{ex:Veronese cubic surface}
	Let $S = \nu_3 (\mathbb{P}^2 ) \subset \mathbb{P}^9$ be the Veronese cubic surface. 
	Then $S^4 = \mathbb{P}^9$ and hence ${\rm ord} (S) = 4$. In particular, if $q$ is a general point of $\mathbb{P}^9$, then $q$ is contained in a $3$-plane spanned by $4$ points of $S$. 
	This implies that the projected surface $S_q$ admits a $4$-secant $2$-plane, and hence
	\[
		\mathrm{index}(S_q) \leq 1
	\]
	(cf. \cite[Theorem 1.1]{EGHP}). 
	By Proposition~\ref{prop:upper semicontinuity}, it follows that $\mathrm{index}(S_q) \leq 1$ for all $q \in \mathbb{P}^9 \setminus S^2$. 
	Also, Theorem 4.1 in \cite{CC13} shows that ${\rm index}(S_q)=0$ for all $q \in S^3 \setminus S^2$ and ${\rm index} (S_q ) \ge 1$ for all $q \in \mathbb{P}^9 \setminus S^3$. 
	Consequently, 
	\begin{equation*}
		{\rm index} (S_q ) = 1
	\end{equation*}
	for all $q \in \mathbb{P}^9 \setminus S^3$.
\end{example} 

Therefore, for a given higher dimensional projective variety $X$, it is an interesting problem to determine when ${\rm index} (X_q)$ behaves as in Example~\ref{ex:projection of scroll}, and when it behaves like the Veronese cubic surface in Example~\ref{ex:Veronese cubic surface}. 
In the latter case, it is also of interest to find a lower bound for ${\rm index} (X_q )$.

\end{document}